\documentclass[13pt,a4paper]{amsart}
\usepackage{fullpage}
\usepackage{mathrsfs}
\usepackage{CJK}
\usepackage{geometry}   
\usepackage{setspace}
\usepackage{amsfonts}
\usepackage{xy}
\usepackage{cite}
\usepackage{amssymb}
\usepackage{color, latexsym}
\usepackage{graphicx}
\usepackage{amsbsy,amscd}
\usepackage{fancyhdr}
\usepackage[colorlinks=true,
           citecolor=blue,
           linkcolor =red]{hyperref}
\xyoption{all}
\allowdisplaybreaks
\makeatletter
\@namedef{subjclassname@2020}{\textup{2020} Mathematics Subject Classification}
\makeatother

\fontsize{9pt}{9pt}\selectfont
\def\bga{\boldsymbol{\alpha}}
\def\bgb{\boldsymbol{\beta}}
\def\bl{\boldsymbol{\ell}}
\def\bmu{\boldsymbol{\mu}}
\def\bgl{\boldsymbol{\lambda}}

\def\bi{\boldsymbol{i}}
\def\bk{\boldsymbol{k}}
\def\bx{\boldsymbol{x}}
\def\by{\boldsymbol{y}}
\def\ci{\mathcal{I}}
\def\ga{\alpha}
\def\gb{\beta}
\def\gl{\lambda}
\newdimen\hoogte    \hoogte=14pt    
\newdimen\breedte   \breedte=14pt   
\newdimen\dikte     \dikte=0.5pt    

\numberwithin{equation}{section}

\swapnumbers
\newtheorem{theorem}[equation]{Theorem}
\newtheorem{lemma}[equation]{Lemma}

\newtheorem{corollary}[equation]{Corollary}

\theoremstyle{definition}

\newtheorem{example}[equation]{Example}
\theoremstyle{remark}
\newtheorem{remark}[equation]{Remark}
\begin{document}
\begin{CJK*}{GBK}{song}
\setlength{\itemsep}{1\jot}
\fontsize{13}{\baselineskip}\selectfont
\setlength{\parskip}{0.3\baselineskip}
\vspace*{0mm}
\title[Characters of Ariki--Koike algebras]{\fontsize{9}{\baselineskip}
\selectfont Characters of Ariki--Koike algebras}
\author{Deke Zhao}
\date{}
\address{\bigskip\hfil\begin{tabular}{l@{}}
          Department of Mathematics\\
            Beijing Normal University at Zhuhai, Zhuhai, 519087\\
             China\\
             E-mail: \it deke@amss.ac.cn \hfill
          \end{tabular}}
\thanks{Supported by Guangdong Basic and Applied Basic Research Foundation
 (Grant No. 2023A1515010251), the National Natural Science Foundation of China (Grant No. 11871107) and the Open Project Program of Key Laboratory of Mathematics and Complex System Beijing Normal University (Grant No. K202402).}
 \dedicatory{Dedicated to Professor Yang Han on the occasion of his 55th birthday}
\vspace*{-3mm}
\begin{abstract}In this paper, we prove the Regev formula for the characters of the Ariki--Koike algebras by applying the Schur--Sergeev reciprocity between quantum superalgebras and Ariki--Koike algebras. As a corollary, we derive the Regev formula for the characters of the complex reflection group of type $G(m,1,n)$, generalizing the Regev  formula for the symmetric groups due to A. Regev (\textit{Israel J. Math.} \textbf{195} (2013): 31--35).
\end{abstract}
\maketitle
\vspace*{-5mm}
\section{Introduction}
In \cite{Regev-2013}, Regev  presented a surprisingly beautiful formula (we refer it as the Regev formula) for the characters of the symmetric group super representations. This formula was obtained by applying the combinatorial theory of Lie superalgebras and the super analogue of the classical Schur--Weyl duality, which was first established  by Sergeev \cite{Serg} and then in more detail by Berele--Regev \cite{B-Regev}. This duality is sometimes
called Schur--Sergeev duality in the literature. In \cite{Z19}, the author  proves a quantum analogue of Regev formula and derives a simple formula for the Iwahori--Hecke algebra super character on the exterior algebra by applying the Schur--Weyl duality between the quantum superalgebras and Iwahori--Hecke algebra of type $A$, which is established independently by Moon \cite{Moon} and Mitsuhashi  \cite{Mit}.

Motivated by these and other works, we \cite{Z24} establish the Schur--Sergeev duality for the Ariki--Koike algebras. It is natural to expect that this duality may help us to understand the characters of the Ariki--Koike algebras. The purpose of this paper is to prove a Regev type formula for the characters of the Ariki--Koike algebras by using the similar arguments as the ones in  \cite{Regev-2013,Z19}.

We now explain our results in detail. Let $q$ and $\boldsymbol{u}=(u_1, \ldots,u_m)$ be indeterminates. The (generic) Ariki--Koike algebra $H_{n}(\boldsymbol{u},q)$ is the associative algebra with identity over the field $\mathbb{K}:=\mathbb{C}(\boldsymbol{u},\sqrt{q})$, generated by  $g_0, g_1, \ldots, g_{n-1}$ with the following relations: \begin{align*}
&(g_0-u_1)\cdots(g_0-u_m)=0,&&\\
&g_0g_1g_0g_1=g_1g_0g_1g_0,&&\\
&g_i^2=(1-q)g_i+q\quad \qquad\text{for }1\leq i< n,&&\\
&g_ig_j=g_jg_i\qquad\qquad\qquad\text{for }|i-j|\geq2, &&\\
 &g_ig_{i+1}g_i=g_{i+1}g_{i}g_{i+1} \quad\text{\, for }1\leq i< n-1.&&\end{align*}
The algebras $H_{m,n}(\boldsymbol{u},q)$ were first constructed by Ariki and
Koike \cite{AK} and classified as cyclotomic Hecke algebras of type $G(m,1,n)$ by Brou\'{e} and Malle \cite{B-Malle}. The quadratic relation is slightly non-standard; it relates to the  standard form by replacing $ g_i$ with $-g_i$. This negated version often yields more elegant $q$-analogues (see e.g. \cite{APR}). The subalgebra $H_n(q)$ of $H_{n}(\boldsymbol{u},q)$ generated by $g_1, \ldots, g_{n-1}$  is the generic Iwahori--Hecke algebra associated with the symmetric group $\mathfrak{S}_n$ of degree $n$.

A composition (resp. partition) $\lambda=(\lambda_1,\lambda_2,\ldots)$ of $n$, denote $\lambda\models n$ (resp. $\lambda\vdash n$), is a sequence (resp. weakly decreasing sequence) of non-negative integers with $|\lambda|=\sum_{i\geq 1}\lambda_i=n$. We define $\ell(\lambda)$ as the length of $\lambda$, i.e., the number of non-zero parts. A partition may also be denoted as $(1^{a_1}\,2^{a_2}\ldots)$, where $a_i$ is the number of parts equal to $i$. An $m$-multicomposition (resp. $m$-multipartition) of $n$ is an $m$-tuple $\boldsymbol{\lambda}=(\lambda^{(1)};\ldots;\lambda^{(m)})$ of compositions (resp. partitions) such that
   \begin{equation*}
   |\bgl|=|\lambda^{(1)}|+\cdots+|\lambda^{(m)}|=n.
 \end{equation*}
Denote by $\mathscr{P}_{m,n}$ the set of all $m$-multipartitions of $n$.  It is known that both the irreducible characters of the complex reflection group $W_{m,n}$ of type $G(m,1,n)$ (in Shephard--Todd's classification \cite{shephard-todd}) and the irreducible representations of $H_{n}(\boldsymbol{u},q)$ are parameterized by $\mathscr{P}_{m,n}$ (see \cite{Macdonald} and \cite{AK}). By slight abuse of notation, we denote  the irreducible character of $W_{m,n}$ and  $H_n(\boldsymbol{u},q)$ associated to $\bgl\in \mathscr{P}_{m,n}$ by $\chi^{\bgl}$.

Given non-negative integers $k_1,\ldots, k_m$, $\ell_1,\ldots, \ell_m$ with $k=k_1+\cdots+k_m$ and $\ell=\ell_1+\cdots+\ell_m$, put $\bk=(k_1,\ldots,k_m)$ and $\bl=(\ell_1,\ldots,\ell_m)$.  Let $\mathbb{Z}_2=\{\bar{0},\bar{1}\}=\mathbb{Z}/2\mathbb{Z}$ and let $V$ be a $\mathbb{Z}_2$-graded $\mathbb{K}$-space with $\mathrm{dim}_{\mathbb{K}}V=k|\ell$. Denote by $\chi_{\bk|\bl;n}^{\boldsymbol{u},q}$ the character of the permutation super representation $(\Phi_{\bk|\bl;n}^{\boldsymbol{u},q}, V^{\otimes n})$  of $H_n(\boldsymbol{u},q)$ (see (\ref{Equ:T0}, \ref{Equ:Super-representaion}) for details).

For each $m$-multipartition $\boldsymbol{\mu}$ of $n$, Mak \cite{Mak} defined the special standard element of $H_{n}(\boldsymbol{u},q)$ associated to $\boldsymbol{\mu}$ and showed that the characters of $H_{n}(\boldsymbol{u},q)$ are determined by their values on these elements  (see also \cite{HR98}). It is natural to determine the character $\chi_{\bk|\bl;n}^{\boldsymbol{u},q}$ via its values on the special standard
elements, while the complexity of the $g_0$-action on $V^{\otimes n}$ makes it is extremely difficult to calculate these values (see (\ref{Equ:T0})). To deal with the obstacle, we use the standard element $g_{\boldsymbol{\mu}}$ of $H_{n}(\boldsymbol{u},q)$ introduced by Shoji in \cite{S} (see (\ref{Equ:T-bmu}) for details).  Thanks to \cite[Proposition~7.5]{S},  $\chi^{\bgl}$ is uniquely determined by its values $\chi^{\bgl}(g_{\bmu})$ on $g_{\bmu}$ for $\bmu\in\mathscr{P}_{m,n}$, so is $\chi_{\bk|\bl;n}^{\boldsymbol{u},q}$.

For a positive integer $a$,  let $\mathscr{C}(a;\bk|\bl)$ be the set of pairs $(\bga;\bgb)$ of $m$-multicompositions such that $|\bga|+|\bgb|=a$, $\ell(\alpha^{(i)})\leq k_i$, and $\ell(\beta^{(i)})\leq \ell_i$ for $1\leq i\leq m$.
For $(\bga;\bgb)\in\mathscr{C}(a;\bk|\bl)$, we denote by $\ell(\bga;\bgb)$ its length, that is,
\begin{equation*}
\ell(\boldsymbol{\alpha};\boldsymbol{\beta}) = \sum_{i=1}^m \left(\ell(\alpha^{(i)}) + \ell(\beta^{(i)})\right).
\end{equation*}

Our first main result  is an explicit formula for $\chi_{\bk|\bl;n}^{\boldsymbol{u},q}$ on the standard elements of $H_n(\boldsymbol{u},q)$, which is  a cyclotomic analogue of \cite[Theorem~A]{Z19}.

\begin{theorem}\label{Them:Char-super-reps}Let $\bmu=(\mu^{(1)};\ldots;\mu^{(m)})$ be an $m$-multipartition of $n$. Then
\begin{equation*}
  \chi_{\bk|\bl;n}^{\boldsymbol{u},q}(g_{\bmu})=  \prod_{r=1}^m\!\prod_{j=1}^{\ell(\mu^{(r)})}
  \sum_{\substack{(\bga;\bgb)\in\mathscr{C}(\mu_j^{(r)};\bk|\bl)}}
   u^{r-1}_{\tilde{\ell}(\bga;\bgb)}(-q)^{|\bga|-\ell(\!\bga\!)}
   (1-q)^{\ell(\bga;\bgb)-1}\prod_{i=1}^m\tbinom{k_i}{\ell(\ga^{(i)})}
     \tbinom{\ell_i}{\ell(\!\gb^{(i)})},
  \end{equation*}
  where $\tilde{\ell}(\bga;\bgb)=\max\{i|\ell(\ga^{(i)},\gb^{(i)})>0\}$.
  \end{theorem}

Let   $\varsigma$ be a primitive $m$-th of unity and $\boldsymbol{\varsigma}=(1, \ldots, \varsigma^{m-1})$. Then $H_{n}(\boldsymbol{\varsigma},1)\cong\mathbb{C}W_{m,n}$. It is well-known that the irreducible character $\chi^{\bgl}$ of $W_{m,n}$ is uniquely determined by its values $\chi^{\bgl}(w_{\bmu})$ on the standard elements $w_{\bmu}$ of $W_{m,n}$ for $\bmu\in\mathscr{P}_{m,n}$ (see (\ref{Equ:w-mu})). Thus we obtain

\begin{corollary}\label{Cor:W-m-n}Let $\bmu=(\mu^{(1)};\ldots;\mu^{(m)})$ be an  $m$-multipartition of $n$. Then
  \begin{equation*}
\chi_{\boldsymbol{k}|\boldsymbol{\ell};n}^{\boldsymbol{\varsigma},1}(w_{\boldsymbol{\mu}}) = \prod_{r=1}^m \prod_{j=1}^{\ell(\mu^{(r)})} \sum_{i=1}^m \left(\ell_i - (-1)^{|\mu_j^{(r)}|}k_i\right) \varsigma^{(r-1)(i-1)}.
\end{equation*}
\end{corollary}

Following Miller \cite{M17},  for $w\in W_{m,n}$  of type $\bmu=(\mu^{(1)}; \ldots; \mu^{(m)})$,  define
\begin{equation*}
  \ell(w)=\text{number of non-zero parts of }\mu^{(1)}=\ell(\mu^{(1)}).
\end{equation*}
A \emph{block characters} of $W_{m,n}$ is a character depends on $w$ only through $\ell(w)$. For example, the Foulkes characters of $W_{m,n}$ introduced by Miller in \cite{M15} are block characters of $W_{r,n}$. Inspired by \cite{GGK}, we yield the following specializations of Corollary~\ref{Cor:W-m-n}, which enables us to understand the Foulkes characters of $W_{m,n}$ by applying the Schur--Sergeev duality in \cite{Z?}.

\begin{corollary}\label{Cor:Block}Let $k$ be a non-negative integer, $\boldsymbol{k}=(k+1,k^{m-1})$, and $\boldsymbol{0}=(0^m)$. Then
  \begin{eqnarray*}
 &&\chi_{\boldsymbol{0}|\boldsymbol{k};n}^{\boldsymbol{\varsigma},1}(w)
  =(mk+1)^{\ell(w)},\\ &&\chi_{\boldsymbol{k}|\boldsymbol{0};n}^{\boldsymbol{\varsigma},1}(w)
  =(-1)^{n+m-1}(-mk-1)^{\ell(w)}
\end{eqnarray*}
for any $w\in W_{m,n}$.
\end{corollary}

For non-negative integer $a$, we define $[a]_{-q}:=\frac{1-(-q)^a}{1+q}$  and write $[a]_{-q}$ as $[a]$ when there is no confusion. A partition $\lambda=(\lambda_1,\lambda_2,\ldots)$ is a \emph{hook partition} if $\lambda_2\leq 1$. Let $\boldsymbol{1}_m=(1^m)$ and  $H(\boldsymbol{1}_m|\boldsymbol{1}_m;n)$ be the set of $m$-multipartitions of $n$ with each component being a hook partition, that is,
 \begin{equation*}
   H(\boldsymbol{1}_m|\boldsymbol{1}_m;n)=\left\{(\lambda^{(1)};\ldots;\lambda^{(m)})
   \vdash n\left|\lambda^{(i)}\text{ is a hook partition for } i=1, \ldots,m\right\}\right..
 \end{equation*}
 For $\bgl\in\mathscr{P}_{m,n}$, let $\#\bgl$ denote the number of its non-zero components. By \cite[Theorem~4.9]{Z24}, $\chi_{\boldsymbol{1}_m|\boldsymbol{1}_m;n}^{\boldsymbol{u},q}$ can be expressed as the sum of the characters of $H(\boldsymbol{u},q)$ labeled by $H(\boldsymbol{1}_m|\boldsymbol{1}_m;n)$ with certain multiplicities. As a consequence, we get the following generalization of \cite[Theorem~B]{Z19}.
\begin{corollary}\label{Cor:mu-1-1}Let $\bmu=(\mu^{(1)};\ldots;\mu^{(m)})$ be an $m$-multipartition of $n$. Then
  \begin{equation*}
  \sum_{\bgl\in H(\boldsymbol{1}_m|\boldsymbol{1}_m;n)}\!\!2^{\#\bgl}\chi^{\bgl}(g_{\bmu})\!=
  \!2^{\ell(\bmu)}\!\prod_{r=1}^m\!\prod_{j=1}^{\ell(\mu^{(r)})}
\sum_{i=1}^m\left([\mu^{(r)}_j]\!+\!(i\!-\!1)\mu_j^{(r)}[\mu_j^{(r)}-1](1\!-\!q)\right)
u_i^{r\!-\!1}\!+\!O((1\!-\!q)^2),
\end{equation*}
where $O((1-q)^2)$ denotes the remainder terms with factor $(1-q)^2$.
\end{corollary}
Furthermore, we obtain a wreath analogue of \cite[Proposition~1.1]{Regev-2013}.
\begin{corollary}\label{Cor:mu-1-1-wreath}Let $\bmu=(\mu^{(1)};\ldots;\mu^{(m)})$ be an $m$-multipartition of $n$. Then
 \begin{equation*}
  \sum_{\bgl\in H(\boldsymbol{1}_m|\boldsymbol{1}_m;n)}2^{\#\bgl}\chi^{\bgl}(w_{\bmu})=
 \left\{\begin{array}{ll}\vspace{2\jot}
 (2m)^{\ell(\mu^{(1)})},&\text{if $\mu^{(1)}\vdash n$ with each part being odd};\\
  0,&\text{otherwise}.
  \end{array}\right.
\end{equation*}
\end{corollary}
 Let $H_n(u,q)$ be the Iwahori--Hecke algebra of type $B_n$ with parameters $u_1=1$ and $u_2=u$. Then we obtain type $B$ analogues of \cite[Theorem~B]{Z19} and \cite[Proposition~1.1]{Regev-2013}.

\begin{theorem}\label{Them:Pair-Regeev}Let\begin{equation*}
  \chi_n:=\frac{1}{2}\sum_{a=0}^{n-1}\left(\chi^{\left((n-a,1^a);\emptyset\right)}+
  \chi^{\left(\emptyset;(n-a,1^a)\right)}\right)+
  \sum_{a=1}^{n-1}\sum_{i=0}^a\sum_{j=0}^{n-a}\chi^{\left((a,1^{a-i});(n-a,1^{n-a-j})\right)}.
\end{equation*}
For any 2-multipartition  $\bmu=(\alpha;\beta)$ of $n$, we have
\begin{align*}
\chi_n(g_{\bmu})\!&=\displaystyle{2^{\ell(\bmu)\!-\!1}\!\prod_{i=1}^{\ell(\alpha)}\prod_{j=1}^{\ell(\beta)}\!\!
\biggl([\alpha_i][\beta_j]\!+\!\left(\alpha_i[\alpha_i\!-\!1][\beta_j]+
[\alpha_i]\beta_j[\beta_j-1]u\right)(1\!-\!q)\biggr)+\!O((1\!-\!q)^2)},
\end{align*}
 and
 \begin{equation*}
  \chi_n(w_{\bmu})=\left\{\begin{array}{ll}\vspace{2\jot}
  \frac{1}{2}(2m)^{\ell(\alpha)},&\text{if $\alpha\vdash n$ with each part being odd};\\
  0,&\text{otherwise}.
  \end{array}\right.
\end{equation*}
\end{theorem}

Let us remark that the Regev formula and its quantum version can be proved by applying the Murnaghan--Nakayama formula for the characters of the symmetric groups and the Iwahori--Hecke algebras of type $A_n$ \cite{Taylor,Z19}.  We hope that the Murnaghan--Nakayama rule for the characters of the Ariki--Koike algebras \cite[Theorem~2.15]{HR98} would provide an alternative proof of the Regev formula for the Ariki--Koike algebras. Note that the Murnaghan--Nakayama rule for Ariki--Koike algebras is formulated in terms of the special standard elements, which are different from the standard elements used here. Recently Jing--Liu \cite{Jing} derived the quantum Regev formula using vertex operators, it is natural to expect a similar approach may be provide a new proof of Theorem~\ref{Them:Pair-Regeev}.

In \cite{LP}, L\"{u}beck--Prasad related characters of $W_{2,n}$ to those of $\mathfrak{S}_{2n}$ via Frobenius's formula. Adin--Roichman \cite{AR} combinatorially generalized this identity. A connection between Regev formulas for the Iwahori--Hecke algebra and the Ariki--Koike algebras, particularly, relating $W_{2,n}$ and  $\mathfrak{S}_{2n}$ is worth exploring (see Corollary~\ref{Cor:alpha-0}).  Finally, it is expected that determining $\chi_{\bk|\bl;n}^{\boldsymbol{u},q}$ via its values on the special standard elements could yield a more compact Regev formula (especially for  type $B$) and  (super) Frobenius formula for the Ariki--Koike algebras.

 This paper is organized as follows. In Section~\ref{Sec:Preliminary},
 we review briefly the permutation super
representation and the Schur--Sergeev duality for the Ariki--Koike algebra. In Sections 3, we introduce the standard elements of $H_n(\boldsymbol{u}, q)$ and prove Theorem~\ref{Them:Char-super-reps} in Section 4.  Section~\ref{Sec:Hook-partitions} reinterprets the character of super permutation representation as a sum (up to certain multiplicities) of characters of the
Ariki--Koike algebra indexed by $(\bk,\bl)$-hook $m$-multipartitions.  Corollaries~\ref{Cor:W-m-n}--\ref{Cor:mu-1-1-wreath} and Theorem~\ref{Them:Pair-Regeev} are proved in the final section.

\section{Preliminaries}\label{Sec:Preliminary}
In this section, we briefly review  the permutation super representation and the Schur--Sergeev duality for the Ariki--Koike algebras.

By a \textit{superspace} over $\mathbb{C}$ we mean a $\mathbb{Z}_2$-graded
$\mathbb{C}$-space $W$ with a decomposition
into two subspaces $W = W_{\bar{0}}\oplus W_{\bar{1}}$.  A non-zero element $x$ of
$W_i$ will be called \textit{homogeneous} and we denote its degree by
$\overline{x}=i\in \mathbb{Z}_2$. We let   $\mathrm{dim\,}_{\mathbb{C}}W:=\mathrm{dim\,}_{\mathbb{C}}W_{\bar{0}}
|\mathrm{dim\,}_{\mathbb{C}}W_{\bar{1}}$ be the dimension of $W$. We will
 view  $\mathbb{C}$ as a superspace concentrated in degree $\bar{0}$.

For $i=1,\ldots,m$, let  $V^{(i)}$  be a superspace over $\mathbb{K}$ with $\dim_{\mathbb{K}} V^{(i)}=k_i|\ell_i$ and let
\begin{equation*}
  \mathfrak{B}^{(i)}=\left\{v_1^{(i)}, \ldots, v_{k_i}^{(i)}, v_{k_i+1}^{(i)}, \ldots, v_{k_i+\ell_i}^{(i)}\right\}
\end{equation*}
be a homogeneous $\mathbb{K}$-basis of $V^{(i)}$ such that
\begin{equation*}
  V_{\bar{0}}^{(i)}=\oplus_{a=1}^{k_i} \mathbb{K}v_a^{(i)}\quad\text{ and }\quad V_{\bar{1}}^{(i)}=\oplus_{a=k_i+1}^{k_i+\ell_i} \mathbb{K}v_a^{(i)}.
\end{equation*}
Clearly the union $\mathfrak{B}=\mathfrak{B}^{(1)}\sqcup\cdots\sqcup\mathfrak{B}^{(m)}$ forms a homogeneous $\mathbb{K}$-basis of $V=V^{(1)}\oplus\cdots\oplus V^{(m)}$.
The vectors in $\mathfrak{B}^{(i)}$ are said to be \textit{color} of $i$. Furthermore, we linearly order the basis vectors by the rule\begin{eqnarray*}&& v_{a}^{(i)}<v_{b}^{(j)}\text{ if and only if either }i<j\text{ or }i=j\text{ and }a<b.                                                          \end{eqnarray*}

For $i=1, \ldots,m$, let $d_i=\sum_{j=1}^i(k_j+\ell_j)$ and relabel the basis vectors as follows: \begin{equation*}
  \begin{array}{cccccccc}
  v_{1}^{(1)}& \cdots & v_{k_1+\ell_1}^{(1)}&\cdots&\cdots& v_{1}^{(m)}&\cdots& v_{k_m+\ell_m}^{(m)}\\
  \updownarrow& \vdots & \updownarrow &\cdots&\cdots& \updownarrow & \vdots & \updownarrow\\\vspace{2\jot}
  v_1& \cdots & v_{d_1}&\cdots&\cdots & v_{d_{m-1}+1} & \cdots& v_{d_m}.
  \end{array}
\end{equation*}
 Then $\mathfrak{B}=\{v_1, \ldots, v_{k+\ell}\}$ is a homogeneous basis of $V$.

Let \begin{align*}\mathcal{I}(k|\ell;n)=\{\bi=(i_1, \ldots, i_n)|1\leq i_t\leq k+\ell, 1\leq t\leq n\}.\end{align*}
For $\bi=(i_1, \ldots, i_n)\in \mathcal{I}(k|\ell;n)$, we write $v_{\bi}=v_{i_1}\otimes \cdots\otimes v_{i_n}$, $\overline{i_j}:=\overline{v_{i_j}}$, and define $c_a(v_{\bi})=b$ if $v_{i_a}$ is being of color $b$. Then $\mathfrak{B}^{\otimes n}=\{v_{\bi}|\bi\in \ci(k|\ell;n)\}$ forms a homogeneous basis of $V^{\otimes n}$, with $\bar{\bi}=\sum_{j=1}^n\overline{i_j}$ and $c_a(\bi)\leq c_{b}(\bi)$ whenever $i_a\leq i_b$.

The Lie superalgebra $\mathfrak{gl}(k_i|\ell_i)$ consists of $(k_i+\ell_i)\times (k_i+\ell_i)$ matrices with $\mathbb{Z}_2$-grading: \begin{eqnarray*} \mathfrak{gl}(k_i|\ell_i)_{\bar{0}}&=&\left\{\left.\left(\begin{array}{cc}\boldsymbol{A} & \boldsymbol{0} \\ \boldsymbol{0} & \boldsymbol{D}\end{array}\right)\right|\boldsymbol{A}\in \mathrm{M}_{k_i\times k_i}(\mathbb{C}), \boldsymbol{D}\in \mathrm{M}_{\ell_i\times \ell_i}(\mathbb{C})\right\},\\ \mathfrak{gl}(k_i|\ell_i)_{\bar{1}}&=&\left\{\left.\left(\begin{array}{cc}\boldsymbol{0}& \boldsymbol{B} \\ \boldsymbol{C} & \boldsymbol{0}\end{array} \right)\right|\boldsymbol{B}\in \mathrm{M}_{k_i\times \ell_i}(\mathbb{C}), \boldsymbol{C}\in \mathrm{M}_{\ell_i\times k_i}(\mathbb{C})\right\}                                      \end{eqnarray*}
and Lie bracket $[\boldsymbol{X},\boldsymbol{Y}]:=\boldsymbol{XY}-(-1)^{\overline{\boldsymbol{X}}
\,\overline{\boldsymbol{Y}}}\boldsymbol{YX}$ for homogeneous $\boldsymbol{X},\boldsymbol{Y}$.

Let $\mathfrak{g}=\mathfrak{gl}(k_1|\ell_1)\oplus\cdots\oplus\mathfrak{gl}(k_m|\ell_m)$ and
$U_q(\mathfrak{g})=U_q(\mathfrak{gl}(k_1|\ell_1))\otimes\cdots\otimes U_q(\mathfrak{gl}(k_m|\ell_m))$,
where $U_q(\mathfrak{gl}(k_i|\ell_i))$ is the quantum enveloping superalgebra of $\mathfrak{gl}(k_i|\ell_i)$ defined in \cite{BKK}. By \cite[(2.15)]{Z24}, $V^{\otimes n}$ is a $U_q(\mathfrak{g})$-modules, we  denote it by  $(\Psi_n, V^{\otimes n})$.

Now we define  $\omega\in\mathrm{End}_{\mathbb{K}}(V)$ and $T, S\in \mathrm{End}_{\mathbb{K}}(V^{\otimes 2})$ as follows:

\begin{align*}
  \omega(v)=u_{c(v)}v,
\end{align*}
\begin{eqnarray}\label{Equ:T}
  &&T(i_1,i_2):=\left\{\begin{array}{ll}\vspace{2\jot}(1-q)(i_1,i_2)
 -(-1)^{\overline{i_1}\,\overline{i_{2}}}\sqrt{q}(i_2,i_1),& \text{if }i_1<i_{2};\\ \vspace{2\jot}
 \frac{(1-q)-(-1)^{\overline{i_1}}(1+q)}{2}(i_1,i_1),& \text{if }i_1=i_{2}; \\
  -(-1)^{\overline{i_1}\,\overline{i_{2}}}\sqrt{q}(i_2,i_1),&\text{if } i_1>i_{2}.
                        \end{array}\right.\end{eqnarray}
\begin{eqnarray*}
  &&S(i_1,i_2):=\left\{\begin{array}{ll}\vspace{2\jot}
 T(i_1,i_2), & \hbox{ if } c(i_1)= c(i_2);\\
 (-1)^{\overline{i_1}\,\overline{i_{2}}}(i_2,i_1), & \hbox{ if } c(i_1)\neq c(i_2).\end{array}\right.\end{eqnarray*}
It is easy to show that $T^2=(q-1)T+q$, which implies $T$ is invertible and
 \begin{eqnarray*}\label{Equ:T^{-1}}
  T^{-1}(i_1,i_2)&=&\left\{\begin{array}{ll}\vspace{2\jot}
 -\frac{(-1)^{\overline{i_1}\,\overline{i_{2}}}}{\sqrt{q}}(i_2,i_1),& \text{if }i_1<i_{2};\\\vspace{2\jot}
 \frac{(q-1)+(-1)^{\overline{i_1}}(1+q)}{2q}(i_1,i_1),& \text{if }i_1=i_{2}; \\
  \frac{q-1}{q}(i_1,i_2)-\frac{(-1)^{\overline{i_1}\,\overline{i_{2}}}}{\sqrt{q}}(i_2,i_1),
  &\text{if } i_1>i_{2}.
   \end{array}\right.\end{eqnarray*}
For $i=1,\ldots,n-1$,  we define \begin{eqnarray*}
\label{Equ:T-def}&&T_i^{\pm1}:=\mathrm{Id}^{\otimes(i-1)}\otimes T^{\pm1}\otimes \mathrm{Id}^{\otimes(n-i-1)},\end{eqnarray*}
\begin{eqnarray*}
\label{Equ:S-def}&&S_i:=\mathrm{Id}^{\otimes(i-1)}\otimes S\otimes \mathrm{Id}^{\otimes(n-i-1)},\end{eqnarray*}
\begin{align*}
\label{Equ:Omega-def}\omega_i&:=\mathrm{Id}^{\otimes(i-1)}\otimes \omega\otimes \mathrm{Id}^{\otimes(n-i)},
\end{align*}
\begin{eqnarray}
\label{Equ:T0}&&T_0:=T_{1}^{-1}\cdots T_{n-1}^{-1}S_{n-1}\cdots S_{1} \omega_1. \end{eqnarray}
 \cite[Theorem~3.6]{Z24} shows that the map
 \begin{equation}\label{Equ:Super-representaion}
   \Phi_{\bk|\bl;n}^{\boldsymbol{u},q}: H_n(\boldsymbol{u},q)\rightarrow \mathrm{End}_{\mathbb{K}}(V^{\otimes n}), \qquad g_i\mapsto T_i, \quad i=0,1, \ldots,n-1,
 \end{equation} is a representation of $H_n(\boldsymbol{u},q)$, which is called \emph{the permutation super representation} of $H_n(\boldsymbol{u},q)$.

\begin{theorem}[\protect{\cite[Theorem~4.9]{Z24}}]\label{Them:Schur-Sergeev}The following Schur--Sergeev duality holds:
  \begin{eqnarray*}
    \Psi_{n}(U_q(\mathfrak{g}))=\mathrm{End}_{H_n(\boldsymbol{u},q)}(V^{\otimes n}) &\quad\text{ and }\quad& \Phi_{\bk|\bl;n}^{\boldsymbol{u},q}(H_n(\boldsymbol{u},q))
    =\mathrm{End}_{U_q(\mathfrak{g})}(V^{\otimes n}).
  \end{eqnarray*}
\end{theorem}

Recall that the complex reflection $W_{m,n}$ of type $G(m,1,n)$ is generated by $s_0, s_1, \ldots, s_{n-1}$  with relations\begin{align*}&s_0^m=1,\quad s_i^2=1&&\text{ for all }1\leq i< n,\\&s_0s_1s_0s_1=s_1s_0s_1s_0,&&\\&s_is_j=s_js_i, &&\text{ for }|i-j|\geq 2,\\ &s_is_{i+1}s_i=s_{i+1}s_{i}s_{i+1}, &&\text{ for }1\leq i< n-1.\end{align*}

 \begin{remark}\label{Rem:Schur-Sergeev-group}Let $\varsigma$ be a primitive $m$-th root of unity and $\boldsymbol{\varsigma}=(1,\varsigma,\ldots,\varsigma^{m-1})$. Then $H_{n}(\boldsymbol{\varsigma},1) \cong\mathbb{C}W_{m,n}$ and $V^{\otimes n}$ is a $\mathbb{C}W_{m,n}$-module:
 \begin{eqnarray*}&&s_0(\boldsymbol{i})=\varsigma^{c_1(\boldsymbol{i})}\boldsymbol{i},\\
&&s_a(\boldsymbol{i})=(-1)^{\overline{i_a}\,\overline{i_{a+1}}}(i_1, \ldots, i_{a-1}, i_{a+1}, i_a, i_{a+2}, \ldots, i_n), a=1,\ldots,n-1.
\end{eqnarray*}
Furthermore, let $U(\mathfrak{gl}(k_i|\ell_i))$ be the enveloping algebra of $\mathfrak{gl}(k_i|\ell_i)$, then there is a Schur--Sergeev duality between $\mathbb{C}W_{m,n}$ and $U(\mathfrak{gl}(k_1|\ell_1))\otimes\cdots\otimes U(\mathfrak{gl}(k_m|\ell_m))$.
\end{remark}

\section{Standard elements}\label{Sec:Standard-elements}
In this section, we recall the standard elements of the generic Ariki--Koike algebra  introduced by Shoji in \cite[\S\S6.1 and 6.7]{S}.

Let $\Delta=\mathrm{det}V(\boldsymbol{u})$ be the determinant of the Vandermonde matrix $V(\boldsymbol{u})=(u_b^a)_{0\leq a<m,1\leq b\leq m}$. Write $V(\boldsymbol{u})^{-1}=\Delta^{-1}V^*(\boldsymbol{u})$ and define a polynomial $
  F_c(x)=\sum_{0\leq i<m}v_{ci}(\boldsymbol{u})x^i$
 for $1\leq c\leq m$.

Now let $H^{s}_n(\boldsymbol{u},q)$ be the associative $\mathbb{K}$-algebra generated by $g_1, \ldots, g_{n-1}$ and $\xi_1, \ldots, \xi_n$ with relations:
\begin{align*}& g_i^2=(1-q)g_i+q,  && 2\leq i\leq n,\\
&(\xi_i-u_1)\cdots(\xi_i-u_m)=0, && 1\leq i\leq n, \\
&g_ig_{i+1}g_{i}=g_{i+1}g_ig_{i+1},&&  1\leq i< n,\\
&g_ig_j=g_jg_i,\qquad  g_j\xi_i=\xi_ig_j, && |i-j|\geq 2, \\
&\xi_i\xi_j=\xi_j\xi_i,&& 1\leq i,j\leq n, \\
&g_j\xi_{j}=\xi_{j-1}g_j+\Delta^{-2}\sum_{a<b}(u_{a}-u_{b})(1-q)
F_{a}(\xi_{j-1})F_{b}(\xi_j),&&2\leq j\leq n, \\
 &g_j\xi_{j-1}=\xi_{j}g_j-\Delta^{-2}\sum_{a<b}(u_{a}-u_{b})
 (1-q)F_{a}(\xi_{j-1})F_{b}(\xi_j),&&2\leq j\leq n.
            \end{align*}
Thanks to \cite[Theorem~3.7]{S}, $H^{s}_n(\boldsymbol{u},q)\cong H_n(\boldsymbol{u},q)$. Hence the  map
  \begin{align*}\widetilde{\Phi}: H^{s}_n(\boldsymbol{u},q)\rightarrow \mathrm{End}_{\mathbb{K}}(V^{\otimes n}),\quad g_i\mapsto T_i,\quad \xi_j\mapsto \omega_j  \end{align*} realizes the  permutation super representation of $H_n(\boldsymbol{u},q)$.

For $a=1,\ldots,n$, let
\begin{align*}
  t_a=s_{a-1}\cdots s_1s_0s_1\cdots s_{a-1}.
\end{align*}Then $\langle t_1,\ldots, t_n\rangle\cong(\mathbb{Z}/m\mathbb{Z})^{n}$, so any element $w\in W_{m,n}\cong(\mathbb{Z}/m\mathbb{Z})^{n}\rtimes \mathfrak{S}_{n}$ decompose uniquely as   $w=t_1^{c_1}\cdots t_n^{c_n}\sigma$ with $\sigma\in \mathfrak{S}_{n}$, $0\leq c_i<m$.

For a composition $\boldsymbol{c}=(c_1, \ldots, c_b)$ of $ n$, let $\mathfrak{S}_{\boldsymbol{c}}\cong\mathfrak{S}_{c_1}\times\cdots
\times\mathfrak{S}_{c_b}$ be the Young subgroup of $\mathfrak{S}_n$. The parabolic subgroup $W_{\boldsymbol{c}}=(\mathbb{Z}/m\mathbb{Z})^{n}\rtimes \mathfrak{S}_{\boldsymbol{c}}$ decomposes as $W_{\boldsymbol{c}}=W^{(1)}\times W^{(2)}\times\cdots\times W^{(b)}$ with  $W^{(i)}\simeq W_{m,c_i}$.

For $a\geq 1$, let $w(i,a)=t_a^{i-1}s_{a-1}\cdots s_1$. Further, for a partition $\mu=(\mu_1,\mu_2,\cdots)$  of $n$, we define
\begin{eqnarray*}
  w(i,\mu)&=&w(i,\mu_1)\times  w(i,\mu_2)\times\cdots.
\end{eqnarray*}
Finally, for an $m$-multipartition $\bmu=(\mu^{(1)}; \ldots; \mu^{(m)})$ of $n$,  let
\begin{eqnarray}\label{Equ:w-mu}
  &&w_{\bmu}=w(1,\mu^{(1)})\cdots w(m,\mu^{(m)}).
\end{eqnarray}
By \cite[4.2.8]{JK}, $\{w_{\bmu}|\bmu\in \mathscr{P}_{m,n}\}$ forms  conjugate classes representatives of $W_{m,n}$.

\vspace{2\jot}
For $\bgl=(\lambda^{(1)};\ldots, \lambda^{(m)})\in \mathscr{P}_{m,n}$, let $p_i=\sum_{a=1}^i|\lambda^{(a)}|$ and let $H_{\bgl^{(i)}}(q)\subset H_n(q)$ be the subalgebra generated by  $g_j$ for $j=p_{i-1}+1, \ldots, p_{i}-1$ with $p_0=0$. Then $H_{\bgl^{(i)}}(q)\cong H_{|\lambda^{(i)}|}(q)$. Define
\begin{equation*}
 H_{\bgl}(q):=H_{\bgl^{(1)}}(q)\otimes\cdots\otimes H_{\bgl^{(m)}}(q)\hookrightarrow H_n(q).
\end{equation*}
Then the subalgebra $H_{\bgl}(\boldsymbol{u},q)\subset H_n(\boldsymbol{u},q)$ generated by $H_{\bgl}(q)$ and $\xi_1, \ldots, \xi_n$ satisfies
\begin{equation}\label{Equ:embedding}
 H_{\bgl}(\boldsymbol{u},q)\cong H_{\lambda^{(1)}}(\boldsymbol{u},q)\otimes\cdots\otimes H_{\lambda^{(m)}}(\boldsymbol{u},q)\hookrightarrow H_n(\boldsymbol{u},q),
\end{equation}
where $H_{\lambda^{(i)}}(\boldsymbol{u},q)\hookrightarrow H_n(u,q)$ is generated by $H_{\lambda^{(i)}}(\boldsymbol{u},q)$ and $\xi_j$ for $j=p_{i-1}+1, \ldots, p_i$.

For $r,a\geq 1$, define
 \begin{equation}\label{Equ:g-r-a}
  g(r,a)=\xi_{a}^{r-1}g_{a-1}\cdots g_1 \text{ with }g(r,1)=\xi_1^{r-1}.
\end{equation}
For a partition $\mu=(\mu_1,\mu_2,\ldots)$ of $n$, let $n_i=\sum_{j=1}^i|\mu_i|$ and define
\begin{equation}\label{Equ:g(r,mu_i)}
  g(r, \mu_i)=\xi^{r-1}_{n_i}g_{n_i-1}\cdots g_{n_i-|\mu_i|} \text{ for }i=1, \ldots, \ell(\mu).
\end{equation}
By using the embedding (\ref{Equ:embedding}),  we set
\begin{equation}\label{Equ:g(r,mu)}
  g(r,\mu)=g(r,\mu_1)\otimes g(r,\mu_2)\otimes\cdots\in H_n(\boldsymbol{u},q).
\end{equation}
For an $m$-multipartition $\bmu=(\mu^{(1)};\ldots; \mu^{(m)})$ of $n$, define the \textit{standard element}:
\begin{eqnarray}\label{Equ:T-bmu}
   &&g_{\bmu}=g(1,\mu^{(1)})\cdots g(m,\mu^{(m)})\in H_n(\boldsymbol{u},q).
\end{eqnarray}
Shoji \cite[Proposition~7.5]{S} showed that the characters of $H_n(\boldsymbol{u},q)$ are uniquely determined  by their values on $g_{\bmu}$ for all $\bmu\in\mathscr{P}_{m,n}$.

 \begin{remark}The definitions of $w_{\bmu}$ and $g_{\bmu}$ are slightly different from the ones in \cite[\S\S6.1 and 6.7]{S}. Indeed $w_{\bmu}$ is the unique special standard elements with the lengths of the blocks with label 0 in descending order introduced in \cite{Mak}, and $g_{\bmu}$ is consistent with the one of $w_{\bmu}$. \end{remark}

\section{Traces of permutation super representations}\label{Sec:Character}
In this section, the character $\chi_{\bk|\bl;n}^{\boldsymbol{u},q}$ of the permutation super representation $(\Phi_{\bk|\bl;n}^{u,q}, V^{\otimes n})$ of $H_n(\boldsymbol{u},q)$ is completely determined by  an explicitly formula for $\chi_{\bk|\bl;n}^{u,q}(g_{\bmu})$ for all standard elements $g_{\bmu}\in H_n(\boldsymbol{u},q)$  indexed by  $\bmu\in\mathscr{P}_{m,n}$.

From now on, we will identify an element $g_w\in H_n(\boldsymbol{u},q)$ with $\Phi_{\bk|\bl;n}^{\boldsymbol{u},q}(g_w)\in \mathrm{End}_{\mathbb{K}}(V^{\otimes n})$ relative to the basis $\mathfrak{B}$, that is, for $w=t_1^{c_1}t_2^{c_2}\cdots t_n^{c_n}s_{n-1}\cdots s_1\in W_{m,n}$, we identify $g_w$ with $\omega^{c_1}_1\cdots \omega_{n}^{c_n}T_{n-1}\cdots T_1$.
Note that the character $\chi_{\Phi_{\bk|\bl;n}^{\boldsymbol{u},q}}$ is determined by  the trace of $g_{w}$ for all ${w\in W_n}$, which reduces to computing  $\mathrm{Trace}(g_{\bmu})$ for all $m$-multipartitions $\bmu$ of $n$.  By  \eqref{Equ:T-bmu}, it  suffices to compute $\mathrm{Trace}(g(r,a),V^{\otimes a})$ for  $1\leq a\leq n$ and $1\leq r\leq m$.

\vspace{2\jot}
For positive integer $a$, we set
\begin{eqnarray*}
         &&\mathscr{C}(a;\bk|\bl)=\left\{(\bga;\bgb)\models a\left|\substack{\vspace{1\jot}\bga=\left(\ga^{(1)},\ldots, \ga^{(m)}\right)
         \text{ with }\ell(\ga^{(i)})\leq k_i \text{ for } i=1,\ldots,m\\ \bgb=\left(\beta^{(1)},\ldots, \beta^{(m)}\right) \text{ with }\ell(\beta^{(i)})\leq \ell_i \text{ for } i=1,\ldots,m}\right.\right\},\\
         &&\mathscr{I}^{+}(a;k|\ell)=\{\bi=(i_1, \ldots, i_a)|1\leq i_1\leq i_2\leq\cdots\leq i_a\leq k+\ell\}.
    \end{eqnarray*}
Recall that $d_i=\sum_{j=1}^i(k_j+\ell_j)$. For $\bi\in\mathscr{I}^{+}(a;k|\ell)$,
   write $\bi$  uniquely as
 \begin{align*}
   \boldsymbol{i}(\bga;\bgb)&=\biggl(1^{\alpha_1^{(1)}},\ldots,d_1^{\beta_{\ell_1}^{(1)}},
   \ldots,   (d_{m-1}+1)^{\alpha_{1}^{(m)}},\ldots,d_m^{\beta_{\ell_m}^{(m)}}\biggr),
    \end{align*}
 for some $(\bga;\bgb)\in\mathscr{C}(a;\bk|\bl)$. Thus $\mathscr{I}^{+}(a;k|\ell)$ is identified with $\mathscr{C}(a;\bk|\bl)$.

\begin{lemma}\label{Lemm:Trace-g(r,a)}For positive integers $r$ and $a$, let $\Theta_{\bk|\bl}^{\boldsymbol{u},q}(r,a)=\mathrm{Trace}(g(r,a),V^{\otimes a})$. Then \begin{equation*}
     \Theta_{\bk|\bl}^{\boldsymbol{u},q}(r,a)=\!\!\!\!
     \sum_{(\bga;\bgb)\in\mathscr{C}(a;\bk|\bl)}
  \!\!\!u^{r-1}_{\tilde{\ell}(\bga;\bgb)}
     (-q)^{|\bga|-\ell(\bga)}(1-q)^{\ell(\bga;\bgb)-1}
     \prod_{i=1}^m\tbinom{k_i}{\ell(\ga^{(i)})}
     \tbinom{\ell_i}{\ell(\gb^{(i)})},
    \end{equation*}
    where $\tilde{\ell}(\bga;\bgb)=\max\{i|\ell(\ga^{(i)};\gb^{(i)})>0\}$.
   \end{lemma}
\begin{proof}For $\boldsymbol{i}\in \mathfrak{B}^{\otimes a}$, let $g(r,a;\boldsymbol{i})$ denote the coefficient of $\bi$ in  $g(r,a)(\boldsymbol{i})$.  Then
    \begin{align*}
         \Theta_{\bk|\bl}^{\boldsymbol{u},q}(r,a)&=
     \sum_{\bi\in\mathfrak{B}^{\otimes n}}g(r,a;\bi).
    \end{align*}
 By (\ref{Equ:T}) and (\ref{Equ:g-r-a}), we get
    \begin{eqnarray}\label{Equ:Ram}
 && g(r,a;\bi)=\left\{\!\!
  \begin{array}{ll}\vspace{1\jot}
  u^{r-1}_{\tilde{\ell}(\bga;\bgb)}(\!-q\!)^{|\bga|-\ell(\bga)}(\!1\!-\!q\!)^{\ell(\bga;\bgb)-1},& \text{if }\bi=\bi(\bga;\bgb)\in\mathscr{C}(a;\bk|\bl);\\
  0,&\text{otherwise.}\end{array}\right.
    \end{eqnarray}
 Given $(\bga;\bgb)\in \mathscr{C}(a;\bk|\bl)$, there are  $\prod_{i=1}^m\tbinom{k_i}{\ell(\!\ga^{(\!i\!)}\!)}\tbinom{\ell_i}{\ell(\!\gb^{(\!i\!)}\!)}$ basis elements  $\bi(\bga;\bgb)$ in $V^{\otimes a}$. Thus
\begin{align*}
        \Theta_{\bk|\bl}^{\boldsymbol{u},q}(r,a)&=\sum_{(\bga;\bgb)\in\mathscr{C}(a;\bk|\bl)}\!\!\!
     u^{r-1}_{\tilde{\ell}(\bga;\bgb)}
     (\!-q\!)^{|\bga|-\ell(\bga)}(\!1\!-\!q\!)^{\ell(\bga;\bgb)-1}
     \prod_{i=1}^m\tbinom{k_i}{\ell(\!\ga^{(\!i\!)}\!)}
     \tbinom{\ell_i}{\ell(\!\gb^{(\!i\!)}\!)}.
        \end{align*}
This completes the proof.
\end{proof}

 \begin{proof}[Proof of Theorem~\ref{Them:Char-super-reps}]By (\ref{Equ:g(r,mu_i)}--\ref{Equ:T-bmu}), we have
 \begin{align*}
  \chi_{\bk|\bl;n}^{\boldsymbol{u},q}(g_{\bmu})&
  \prod_{r=1}^{m}\mathrm{Trace}(g(r,\mu^{(r)}),V^{\otimes |\mu^{(r)}|})=\prod_{r=1}^{m}\prod_{j=1}^{\ell(\mu^{(r)})} \Theta_{\bk|\bl}^{\boldsymbol{u},q}(r,\mu_j^{(r)}).
    \end{align*}
Thus Lemma~\ref{Lemm:Trace-g(r,a)} shows
 \begin{align*}
  \chi_{\bk|\bl;n}^{\boldsymbol{u},q}(g_{\bmu})&\prod_{r=1}^m\!\!\prod_{j=1}^{\ell(\!\mu^{(\!r\!)}\!)}
 \sum_{\substack{(\bga;\bgb)
    \in\mathscr{C}(\mu_j^{(\!r\!)};\bk|\bl)}}
   \!\!\!\!
   u^{r\!-\!1}_{\tilde{\ell}(\bga;\bgb)}(\!-q\!)^{|\bga|-\ell(\!\bga\!)}
   (\!1\!-\!q\!)^{\ell(\!\bga;\bgb\!)-1}\!\!\prod_{i=1}^m\!\!\tbinom{k_i}{\ell(\!\ga^{(\!i\!)}\!)}
     \tbinom{\ell_i}{\ell(\!\gb^{(\!i\!)}\!)}. \end{align*}
     The proof is completed.
 \end{proof}
\begin{remark}If $m=1$ and let $u_1=1$, then $(\Phi^{1,q}_{k|\ell}, V^{\otimes n})$ is the sign $q$-permutation representation of $H_n(q)$ and we reprove \cite[Theorem~A]{Z19}. Note: there are minor misprints in the cited work. When $m=1$, $u_1=1$, and $\ell=0$, we obtain a negated version of \cite[Theorem~4.1 and Proposition~4.2]{Ram}): for any partition $\mu$ of $n$,
\begin{eqnarray*}
    \chi_{k|0;n}^{1,q}(\mu)&=&\prod_{i=1}^{\ell(\mu)}
    \sum_{(\ga;\emptyset)\in\mathscr{C}(\mu_i;k|0)}\tbinom{k}{\ell(\ga)}
    (-q)^{\mu_i-\ell(\ga)}(1-q)^{\ell(\ga)-1}.
   \end{eqnarray*}
\end{remark}
\begin{remark}For $\boldsymbol{\ell}=\boldsymbol{0}$, we recover the negated version of \cite[Theorem~6.8]{S}: for any $m$-multipartition $\bmu$ of $n$,
  \begin{eqnarray*}
    \chi_{\bk|\boldsymbol{0};n}^{m,q}(\bmu)&=&\prod_{r=1}^m\!\!
    \prod_{j=1}^{\ell(\mu^{(r)})}\!\!\!\sum_{\substack{(\bga;\boldsymbol{\emptyset})
    \in\mathscr{C}(\mu_j^{(\!r\!)};\bk|\boldsymbol{0})}}
   \!\!\!\!\!\!\!\negmedspace
   u^{r\!-\!1}_{\tilde{\ell}(\bga)}(-q)^{|\bga|-\ell(\!\bga\!)}
   (1-q)^{\ell(\!\bga\!)-1}\!\!\prod_{i=1}^m\!\!\tbinom{k_i}{\ell(\!\ga^{(\!i\!)}\!)}.
   \end{eqnarray*}
\end{remark}

The following fact clarifies the relation between (\ref{Equ:Ram}) and \cite[Lemma~4.12]{Ram}.
For any partition $\lambda$, let $s_{\lambda}(t)$ be the Schur function in variable $t$. The \textit{duality formula} for Schur function (see \cite[Page~43]{Macdonald}) states that
\begin{equation*}
  s_{\lambda}(-t)=(-1)^{|\lambda|}s_{\lambda'}(t),
\end{equation*}
where $\lambda'$ denotes the conjugate partition to $\lambda$. By applying \cite[Lemma~4.12]{Ram}, we have
\begin{equation*}
  s_{\lambda}(q)=\left\{\begin{array}{ll}\vspace{1\jot}
  q(q-1)^{|\lambda|-b},& \text{if $\lambda=(1^{|\lambda|-b}b)$ for some $b\geq 1$};\\
   0,&\text{otherwise}.
  \end{array}\right.
\end{equation*}

\begin{corollary}Keeping notations as above. Then
  \begin{eqnarray*}
  && g(r,a;\bi)=\left\{\!\!
  \begin{array}{ll}\vspace{1\jot}
  -u^{r-1}_{\tilde{\ell}(\bga;\bgb)}s_{(1^{|\bga|-\ell(\bga)-1}\ell(\bga)+1)}(\!1\!-\!q\!)
  s_{(1\,\ell(\bga;\bgb)-1)}(\!-q\!),& \text{if }\bi=\bi(\bga;\bgb)\in\mathscr{C}(a;\bk|\bl);\\
  0,&\text{otherwise.}\end{array}\right.
  \end{eqnarray*}
\end{corollary}
\begin{proof}By applying \cite[Lemma~4.12]{Ram},
 \begin{eqnarray*} g(r,a;\bi)&=&
  u^{r-1}_{\tilde{\ell}(\bga;\bgb)}(-q)(1-q)^{\ell(\bga;\bgb)-2}
  s_{(1^{|\bga|\!-\!\ell(\bga)\!-\!1}\ell(\bga)+1)}(\!1\!-\!q\!)\\
  &=&(-1)^{\ell(\bga;\bgb)-1}u^{r-1}_{\tilde{\ell}(\bga;\bgb)}s_{(1^{\ell(\bga;\bgb)-2}2)}(q)
  s_{(1^{|\bga|\!-\!\ell(\bga)\!-\!1}\ell(\bga)+1)}(\!1\!-\!q\!)\\
  &=&-u^{r-1}_{\tilde{\ell}(\bga;\bgb)}s_{(1\,\ell(\bga;\bgb)-1)}(-q)
  s_{(1^{|\bga|\!-\!\ell(\bga)\!-\!1}\ell(\bga)+1)}(\!1\!-\!q\!),
  \end{eqnarray*}
  where the last equality follows by applying the duality formula for Schur functions
\end{proof}

\section{$(\bk,\bl)$-hook multipartitions}\label{Sec:Hook-partitions}
In this section,  combining the Schur--Sergeev duality between the quantum superalgebra and Ariki--Koike algebra \cite{Z24} and the representation theory of quantum superalgebra
$U_q(\mathfrak{g})$, we reinterpret the trace of permutation super representation as a sum of irreducible characters indexed by $(\bk,\bl)$-hook multipartitions of $n$.

A partition $\gl=(\gl_1, \gl_2, \cdots)\vdash n$ is said to be a \textit{$(k, \ell)$-hook partition} of $n$ if $\gl_{k+1}\leq \ell$. We let  $H(k|\ell;n)$ denote the set of all such partitions:
\begin{eqnarray}\label{Def:hook}
H(k|\ell;n)=\{\gl=(\gl_1,\gl_2,\cdots)\vdash n\mid \gl_{k+1}\le \ell\}.
\end{eqnarray}
An $m$-multipartition $\bgl$ of $n$ is a \textit{$(\bk,\bl)$-hook $m$-multipartition} if $\lambda^{(i)}$ is a $(k_i,\ell_i)$-hook partition for $i=1, \ldots, m$. Let $H(\bk|\bl;n)$ denote the set of all $(\bk,\bl)$-hook $m$-multipartitions of $n$. By \cite{BKK},   the irreducible $U_q(\mathfrak{gl}(k|\ell))$-modules in $V^{\otimes n}$ are parameterized by  $(k,\ell)$-hook partitions of $n$. As a consequence,  the irreducible $U_q(\mathfrak{g})$-modules in $V^{\otimes n}$ are parameterized by $H(\bk|\bl;n)$.

 Define commuting variables sets  $\boldsymbol{x}, \boldsymbol{y}$ of $k+\ell$  as:
\begin{eqnarray*}\label{Equ:Variables}&&\begin{array}{c}\vspace{2\jot}
\boldsymbol{x}^{(i)}=\left\{x^{(i)}_1,\ldots,x_{k_i}^{(i)}\right\},\quad 1\leq i\leq m,\\\vspace{2\jot}
\boldsymbol{y}^{(i)}=\left\{y^{(i)}_1,\ldots,y_{\ell_i}^{(i)}\right\},\quad1\leq i\leq m,\\ \boldsymbol{x}=\boldsymbol{x}^{(1)}\cup\cdots\cup\boldsymbol{x}^{(m)},
\quad\boldsymbol{y}=\boldsymbol{y}^{(1)}\cup\cdots\cup\boldsymbol{y}^{(m)}.\end{array}
\end{eqnarray*}
We linearly order variables by the rule\begin{eqnarray*}  x_{a}^{(i)}<x_{b}^{(j)}&\text{ if and only if }& \text{either }i<j\text{ or }i=j\text{ and }a<b,\\ y_{a}^{(i)}<y_{b}^{(j)}&\text{ if and only if }&\text{either } i<j\text{ or }i=j\text{ and }a<b.                                                         \end{eqnarray*}

Recall that the \textit{diagram}  of an $m$-multipartition $\bgl$ of $n$ is
\begin{equation*}
\bgl:=\{(i,j,c)\in\mathbb{Z}_{>0}\times\mathbb{Z}_{>0}\times \{1,\ldots,m\}|1\le j\le\lambda^c_i\}.
\end{equation*}
By a  \textit{$\bgl$-tableau}, we mean a filling of its diagram  with variables  $\bx$, $\by$. We say that a tableau $\boldsymbol{T}=(\boldsymbol{T}^{(1)};\ldots;\boldsymbol{T}^{(m)})$ is \emph{$\bk|\bl$-semistandard} if  for each $i=1,\ldots,m$:
\begin{enumerate}
  \item[(a)]$\boldsymbol{T}^{(i)}$ is a $\lambda^{(i)}$-tableau containing variables  $\bx^{(i)}, \by^{(i)}$ and its $\bx$ part (the boxes filled with variables $\bx$ of $\boldsymbol{T}^{(i)}$) is a tableau and its $\by$ part is a skew tableau;
  \item[(b)] The $\bx$ part is nondecreasing in rows, strictly increasing in columns;
  \item[(c)] The $\by$ part is nondecreasing in columns, strictly increasing in rows.
\end{enumerate}

Let $\mathrm{std}_{\bk|\bl}(\bgl)$ denote the $\bk|\bl$-semistandard tableaux of shape $\bgl$ and $s_{\bk|\bl}(\bgl)$ its cardinality. By \cite[\S\,2]{B-Regev} and \cite[Lemma~4.2]{BKK}, $s_{\bk|\bl}(\bgl)\neq 0$ if and only if $\bgl\in H(\bk|\bl;n)$.

By \cite{BKK}, the irreducible summands $V_{\lambda}$, $\lambda\in H(k|\ell;n)$, in the $U_q(\mathfrak{gl}(k|\ell))$-modules $(\Psi^{\otimes n}, V^{\otimes n})$ have a basis parameterized by the  $(k,\ell)$-semistandard tableaux of shape $\gl$, which means $\dim _{\mathbb{K}}V_{\lambda}=s_{k|\ell}(\gl)$. Thus the irreducible summand $V_{\bgl}$ in the $U_q(\mathfrak{g})$-module $(\Psi_n, V^{\otimes n})$ has a basis parameterized by the $\bk|\bl$-semistandard tableaux of shape $\bgl$ and $\dim_{\mathbb{K}}V_{\bgl}=s_{\bk|\bl}(\bgl)$. Furthermore, by \cite[Theorem~4.9]{Z24}, we have the following $U_q(\mathfrak{g})$-$H_n(\boldsymbol{u},q))$-bimodule isomorphism: \begin{equation*}
  V^{\otimes n} \cong \bigoplus_{\bgl\in H(\bk|\bl;n)}V_{\bgl}\otimes S^{\bgl},
\end{equation*}
where $S^{\bgl}$ is the irreducible representation of $H_n(\boldsymbol{u},q)$ indexed by $\bgl$. Thus
\begin{equation}\label{Equ:Trace-char}
   \chi_{\bk|\bl;n}^{\boldsymbol{u},q}=\sum_{\bgl\in H(\bk|\bl;n)}s_{\bk|\bl}(\bgl)\chi^{\bgl}.
\end{equation}
Combining Theorem~\ref{Them:Char-super-reps} and (\ref{Equ:Trace-char}) yields:
\begin{corollary}\label{Cor:Regev-formula}For $\bmu=(\mu^{(1)};\ldots;\mu^{(m)})\vdash n$,
  \begin{equation*}
  \sum_{\bgl\in H(\bk|\bl;n)}\!\!\!\!\!\!s_{\bk|\bl}(\!\bgl\!)\chi^{\bgl}(\!g_{\bmu}\!)\!=
  \!\prod_{r=1}^m\!\prod_{j=1}^{\ell(\!\mu^{(\!r\!)}\!)}\!\sum_{\substack{(\bga;\bgb)
    \in\mathscr{C}(\mu_j^{(r)};\bk|\bl)}}
  \!\!\!\!\!\!\!\!u^{r\!-\!1}_{\tilde{\ell}(\bga;\bgb)}(\!-q\!)^{|\bga|\!-\!\ell(\bga)}
   (\!1\!-\!q\!)^{\ell(\!\bga;\bgb\!)\!-\!1}\!\prod_{i=1}^m\!
   \tbinom{k_i}{\ell(\!\ga^{(\!i\!)}\!)}\!\tbinom{\ell_i}{\ell(\!\gb^{(\!i\!)}\!)},
\end{equation*}
 where $\tilde{\ell}(\bga;\bgb)=\max\{i|\ell(\ga^{(i)},\gb^{(i)})>0\}$.
\end{corollary}

\section{Specializations}\label{Sec:Specializations}
This section investigates various specializations of Theorem~\ref{Them:Char-super-reps}. Specifically,  we first derive a formula for characters of complex reflection group $W_{m,n}$ by setting $\boldsymbol{u}=\boldsymbol{\varsigma}$ and $q=1$. Then we prove an  Ariki--Koike analogue of \cite[Theorem~B]{Z19} for $\bk=\bl=\boldsymbol{1}_m$, and prove  type $B$-analogues of \cite[Proposition~1.1]{Regev-2013} and \cite[Theorem~B]{Z19}.

 \begin{proof}[Proof of Corollary~\ref{Cor:W-m-n}]For $r=1,\ldots,m$ and $j=1, \ldots, \ell(\mu^{r})$, let  $\mathscr{C}(\mu_j^{(r)};\bk)$ (resp. $\mathscr{C}(\mu_j^{(r)};\bl)$) denote the set of $m$-multicompositions $\bga=(\alpha^{(1)}; \ldots, \alpha^{(m)})$ of $\mu_j^{(r)}$ with $\ell(\alpha^{(i)})\leq k_i$ (resp. $\ell(\alpha^{(i)})\leq \ell_i$). By applying Lemma~\ref{Lemm:Trace-g(r,a)}, we yield
  \begin{align*}
 \Theta_{\bk|\bl}^{\boldsymbol{\varsigma},1}(r,\mu_j^{(r)})&=\sum_{\substack{(\bga;\bgb)
    \in\mathscr{C}(\mu_j^{(r)};\bk|\bl)\\\ell(\bga;\bgb)=1}}
   \!\!\varsigma^{(r-1)(\tilde{\ell}(\bga;\bgb)-1)}(-1)^{|\bga|-\ell(\bga)}
   \prod_{i=1}^m\tbinom{k_i}{\ell(\ga_i)}
     \tbinom{\ell_i}{\ell(\gb_i)}.\end{align*}
Splitting the right hand side into $\bga$ and $\bgb$ components, we have
\begin{align*}\mathrm{RHS}&=\sum_{\substack{\bga
    \in\mathscr{C}(\mu_j^{(r)};\bk)\\\ell(\bga)\!=\!1}}
   \!\!\varsigma^{(r-1)(\tilde{\ell}(\bga)-1)}   (-\!1)^{|\bga|\!-\!\ell(\bga)}\prod_{i=1}^m\tbinom{k_i}{\ell(\ga_i)}\!+\!\sum_{\substack{\bgb
    \in\mathscr{C}(\mu_j^{(r)};\bl)\\\ell(\bgb)=1}}
   \!\!\varsigma^{(r\!-\!1)(\tilde{\ell}(\bgb)\!-\!1)}
   \prod_{i=1}^m\tbinom{\ell_i}{\ell(\gb_i)}\\
   &=\sum_{i=1}^m(\ell_i-(-1)^{|\mu_j^{(r)}|}k_i)\varsigma^{(r-1)(i-1)}.
   \end{align*}
Thus the corollary follows  by applying Theorem~\ref{Them:Char-super-reps}.
 \end{proof}

\begin{proof}[Proof of Corollary~\ref{Cor:Block}]As $\chi_{\boldsymbol{0}|\boldsymbol{k};n}^{\boldsymbol{\varsigma},1}$ and $\chi_{\boldsymbol{k}|\boldsymbol{0};n}^{\boldsymbol{\varsigma},1}$
   are class functions on $W_{r,n}$, we may assume $w$  is $w_{\bmu}$ for some $m$-multipartition $\bmu$ of $n$. Then Corollary~\ref{Cor:W-m-n} shows
  \begin{eqnarray*}\vspace{2\jot}
  \Theta_{\boldsymbol{0}|\bk}^{\boldsymbol{\varsigma},1}(r,\mu_j^{(r)})
  &=&\left\{\begin{array}{ll}
  mk+1,&\text {if $r=1$},\\1,&\text{otherwise};
     \end{array}\right.\\
 \Theta_{\bk|\boldsymbol{0}}^{\boldsymbol{\varsigma},1}(r,\mu_j^{(r)})
  &=&\left\{\begin{array}{ll}
  -(-1)^{|\mu_j^{(1)}|}mk+1,&\text {if $r=1$},\\(-1)^{|\mu_j^{(r)}|+1},&\text{otherwise}.
     \end{array}\right.
  \end{eqnarray*}
The corollary follows directly.
\end{proof}
 Combining Remark~\ref{Rem:Schur-Sergeev-group}, Corollaries~\ref{Cor:Regev-formula}, and \ref{Cor:W-m-n}, we obtain
\begin{corollary}\label{Cor:W-R}For $\bmu=(\mu^{(1)};\ldots;\mu^{(m)})\vdash n$,
  \begin{equation*}
  \sum_{\bgl\in H(\bk|\bl;n)}s_{\bk|\bl}(\bgl)\chi^{\bgl}(w_{\bmu})=
  \!\prod_{r=1}^m\prod_{j=1}^{\ell(\mu^{(r)})}
  \sum_{i=1}^m\left(\ell_i-(-1)^{|\mu_j^{(r)}|}k_i\right)\varsigma^{(r-1)(i-1)}.
\end{equation*}
\end{corollary}

For $a\geq 1$, $1\leq i\leq m$, and $j\geq 1$, let  \begin{align*}  \mathscr{C}_j(a;i)&=\left\{(\alpha;\beta)\models a\left|\begin{array}{l}\alpha=(a_1, \ldots,a_i), \beta=(b_1, \ldots, b_i),\ell(\alpha;\beta)=j,\\a_s,b_s\geq 0\text{ for }1\leq s\leq i-1, \text{ and } a_i+b_i>0 \end{array}\right\}\right., \end{align*}
and define
\begin{equation*}
\Theta_j(a;i)=\sum_{\substack{(\ga;\gb)\in\mathscr{C}_j(a;i)}}(-q)^{|\ga|-\ell(\ga)}(1-q)^{\ell(\ga;\gb)-1}.
\end{equation*}
Thanks to Lemma~\ref{Lemm:Trace-g(r,a)}, we can
express $\Theta_{\boldsymbol{1}_m|\boldsymbol{1}_m}^{\boldsymbol{u},q}(r,a)$ as:
\begin{equation*}
\Theta_{\boldsymbol{1}_m|\boldsymbol{1}_m}^{\boldsymbol{u},q}(r,a)
=\sum_{i=1}^mc_i(r,a)u_i^{r-1}.
\end{equation*}
Then  $c_i(r,a)=\sum_{j=1}^{2i}\Theta_j(a;i)$. In particular \begin{align*}
c_1(r,a)&=\Theta_1(a;1)+\Theta_2(a;1)=1+(-q)^{a-1}+(1-q)[a-1]=2[a].\end{align*}

It is interesting to give an explicit formula of $c_i(r,a)$ ($i=1, \ldots, m$), which enables us to present a closed formula for the characters of the Ariki--Koike algebras. The following examples  present explicit formulas of $\Theta_1(a;i)$, $\Theta_2(a;i)$, and $c_2(r,a)$.
\begin{example}
Let $\boldsymbol{e}_s$ be the $i$-dimensional row vector with  the $s$-th coordinate being 1 and zero others. Then $\mathscr{C}_1(a;i)=\left\{(a\boldsymbol{e}_i;\emptyset), (\emptyset; a\boldsymbol{e}_i)\right\}$ and
\begin{align*}\mathscr{C}_2(a;i)&=\left\{(x\boldsymbol{e}_s+(a-x)\boldsymbol{e}_i;\emptyset), (\emptyset; x\boldsymbol{e}_s+(a-x)\boldsymbol{e}_{i})|1\leq x<a, 1\leq s<i\right\}\\
&\qquad\cup\left\{(x\boldsymbol{e}_s;(a-x)\boldsymbol{e}_i),
(x\boldsymbol{e}_i;(a-x)\boldsymbol{e}_s) |1\leq x<a, 1\leq s\leq i\right\}.\end{align*}
Thus
 \begin{align*}\Theta_1(a;i)&=1+(-q)^{a-1},\\ \Theta_2(a;i)&=(i-1)(a-1)(1-q)(1+(-q)^{a-2})+(2i-1)(1-q)[a-1].\end{align*}
\end{example}

\begin{lemma}\label{Lemm:Z12}For integer $a\geq 1$ and $b\geq0$, set $Z_b(a):=\displaystyle\sum_{z=0}^{a-1}\frac{(z+b)!}{z!}(-q)^{z}$.  Then
\begin{align*}
  Z_b(a)&=\frac{bZ_{b-1}(a)-(-q)^a\prod_{i=0}^{b-1}(a+i)}{1+q}.
  \end{align*}\end{lemma}
  \begin{proof}Since  $-qZ_b(a)=\sum_{z=1}^{a}z(z+1)\cdots(z+b-1)(-q)^{z}$, we get
\begin{eqnarray*}
  (1+q)Z_b(a)&=&bZ_{b-1}(a)-(-q)^{a}\prod_{i=0}^{b-1}(a+i),
\end{eqnarray*}
which proves the lemma.\end{proof}

\begin{example}For integer $a\geq 1$ and $b\geq0$, we have
\begin{align*}\mathscr{C}_3(a;2)&=\left\{((a_1,a_2);(a_3,a_4))\models a|\text{exactly one $a_i$ is zero}\right\},\\
\mathscr{C}_4(a;2)&=\left\{((a_1,a_2);(a_3,a_4))\models a|a_j\geq 1\right\},\\
\Theta_3(a;2)&=2(a-1)(1-q)^2[a-2],\\ \Theta_4(a;2)&=(a-1)[a-3](1-q)^3-Z_2(a-3)(1-q)^3.
\end{align*}
Thus
\begin{align*}c_2(r,a)&=\Theta_1(a;2)+\Theta_2(a;2)+\Theta_3(a;3)+\Theta_4(a;2)\\
    &=2[a]\!+\!2a[a\!-\!1](1\!-\!q)\!+  \!(a\!-\!1)(1\!-\!q)^2\left([a\!-\!2]\!+\![a\!-\!3](1\!-\!q)-Z_2(a\!-\!3)(1\!-\!q)\right).
\end{align*}
\end{example}

\begin{lemma}For  $a,i,r\geq 1$ with $1\leq i,r\leq m$,
 \begin{equation*}
 c_i(r,a)=2[a]+2a(i-1)(1-q)[a-1]+O((1-q)^2),
\end{equation*}
where $O((1-q)^2)$ denotes the remainder terms with factor $(1-q)^2$.
\end{lemma}
\begin{proof}Note that $\Theta_j(i,a)=O((1-q)^2)$ for $j\geq 3$ and that
\begin{equation*}
  1+(-q)^{a-1}+(1-q)[a-1]_{-q}=2[a].
\end{equation*}
Thus
  \begin{align*}c_i(r,a)&=\Theta_1(i,a)+\Theta_2(i,a)+O((1-q)^2)\\
&=1\!+\!(-q)^{a\!-\!1}\!+\!(i\!-\!1)(a\!-\!1)(1\!-\!q)(1+(\!-\!q)^{a\!-\!2})\!
+\!(2i\!-\!1)(1\!-\!q)[a\!-\!1]\!+\!O((1-q)^2)\\
    &=2[a]\!+\!(i\!-\!1)(a\!-\!1)(1\!-\!q)(1+(\!-\!q)^{a\!-\!2})\!
+\!2(i\!-\!1)(1\!-\!q)[a\!-\!1]\!+\!O((1-q)^2)\\
    &=2[a]\!+\!2a(i\!-\!1)(1\!-\!q)[a-1]\!+\!O((1-q)^2).
\end{align*}
It completes the proof.
\end{proof}

\begin{proof}[Proof of Corollary~\ref{Cor:mu-1-1}]
 By \cite[Theorem~6.24]{B-Regev}, $s_{\boldsymbol{1}_m|\boldsymbol{1}_m}(\bgl)=2^{\#\bgl}$ for $\bgl\in H(\boldsymbol{1}_m|\boldsymbol{1}_m;n)$. Since $\mathscr{C}(a; \boldsymbol{1}_m|\boldsymbol{1}_m)$ consists composition pairs   $(\alpha;\beta)\models a$ with $0\leq \ell(\alpha),\ell(\beta)\leq m$, Corollary~\ref{Cor:Regev-formula} gives
 \begin{align*}
 \sum_{\bgl\in H(\boldsymbol{1}_m|\boldsymbol{1}_m;n)}2^{\#\bgl}\chi^{\bgl}(g_{\bmu})
 &=2^{\ell(\bmu)}\!\prod_{r=1}^m\!\prod_{j=1}^{\ell(\mu^{(r)})}
\!\!\sum_{i=1}^m\left(\![\mu^{(r)}_j]\!+\!(i\!-\!1)\mu_j^{(r)}[\mu_j^{(r)}\!-\!1](1\!-\!q)\!\right)
u_i^{r\!-\!1}\!+\!O((1\!-\!q)^2).
 \end{align*}
It completes the proof.
\end{proof}

\begin{proof}[Proof of Corollary~\ref{Cor:mu-1-1-wreath}]By Corollary~\ref{Cor:W-m-n},
  \begin{eqnarray*}
  \sum_{\bgl\in H(\boldsymbol{1}_m|\boldsymbol{1}_m;n)}2^{\#\bgl}\chi^{\bgl}(\bmu)&=&
  \!\prod_{r=1}^m\prod_{j=1}^{\ell(\mu^{(r)})}
  \sum_{i=1}^m\left(1-(-1)^{|\mu_j^{(r)}|}\right)\varsigma^{(r-1)(i-1)}\\
  &=&\prod_{j=1}^{\ell(\mu^{(1)})}m
  \left(1-(-1)^{|\mu_j^{(1)}|}\right)\\
  &=&\left\{\begin{array}{ll}\vspace{2\jot}
 (2m)^{\ell(\bmu)},&\text{if $\mu^{(1)}\vdash n$ with each part being odd};\\
  0,&\text{otherwise},
  \end{array}\right.
\end{eqnarray*}
where the second equality follows by applying $\sum_{i=1}^m\varsigma^{t(i-1)}=0$ for  $1\leq t<m$.
\end{proof}

\begin{proof}[Proof of Theorem~\ref{Them:Pair-Regeev}]Clearly, if  $(\mu;\nu)$ is a pair of  hook partitions with $|\mu|+|\nu|=n$, then $\mu=(x,1^i)$ and $\nu=(y, 1^j)$, where $x,y,i,j$ are non-negative integers with $x+y+i+j=n$ and $x, y\geq 1$ when $i,j>0$. Therefore $\chi_n=\sum_{\bgl\in H(\boldsymbol{1}_2|\boldsymbol{1}_2,n)}2^{\#\bgl}\chi^{\bgl}$ and the theorem follows directly by applying Corollaries~\ref{Cor:mu-1-1} and \ref{Cor:mu-1-1-wreath}.
\end{proof}

The following fact follows directly by applying Theorem~\ref{Them:Pair-Regeev}.

\begin{corollary}\label{Cor:alpha-0}Keeping notations as Theorem~\ref{Them:Pair-Regeev}. Then for $\alpha\vdash n$, we have
  \begin{align*}
    \chi_n(g_{(\alpha;\emptyset)})=\chi_n(g_{(\emptyset;\alpha)})=
    2^{\ell(\alpha)}\prod_{i=1}^{\ell(\alpha)}\biggl([\alpha_i]+
    \alpha_i[\alpha_i-1]u(1-q)\biggr)+O((1\!-\!q)^2).
  \end{align*}
\end{corollary}

\subsection*{Acknowledgements.} The author is very grateful to the anonymous referee  for their  numerous helpful comments and corrections, which greatly improved
the presentation of this paper.


\end{CJK*}
\end{document}